\documentclass[reqno]{amsart}
\usepackage{amssymb}
\usepackage{hyperref}
\usepackage{booktabs}
 \usepackage{multirow}

 \usepackage{color}

\newtheorem{theorem}{Theorem}
\newtheorem{proposition}[theorem]{Proposition}

\theoremstyle{remark}
\newtheorem{remark}[theorem]{Remark}

\numberwithin{equation}{section}

\begin{document}

\title[Cubature Rules for  Unitary Jacobi Ensembles]
{Cubature Rules for Unitary Jacobi Ensembles}

\author{J.F.  van Diejen}

\address{
Instituto de Matem\'atica y F\'{\i}sica, Universidad de Talca,
Casilla 747, Talca, Chile}

\email{diejen@inst-mat.utalca.cl}

\author{E. Emsiz}

\address{
Delft Institute of Applied Mathematics,
Delft University of Technology,
Van Mourik Broekmanweg 6, 2628 XE, Delft, The Netherlands}
\email{e.emsiz@tudelft.nl}

\subjclass[2010]{Primary: 65D32;  Secondary 15B52, 28C10, 43A75}
\keywords{cubature rules, random matrices, compact Lie groups, Haar measures}

\thanks{This work was supported in part by the {\em Fondo Nacional de Desarrollo
Cient\'{\i}fico y Tecnol\'ogico (FONDECYT)} Grant   \# 1170179.}

\date{June 2020}

\begin{abstract}
We present Chebyshev type cubature rules for the exact integration of rational symmetric functions with poles on prescribed coordinate hyperplanes. Here the integration is with respect to the densities of unitary Jacobi ensembles
stemming from the Haar measures of the orthogonal and the compact symplectic Lie groups. 
\end{abstract}

\maketitle



\section{Introduction}\label{sec1}
It has been long known that orthogonal polynomials play a pivotal role in the construction of quadrature rules for the efficient numerical integration of functions over the interval
\cite{sze:orthogonal}. Part of the underlying theory has been generalized to higher dimensions, resulting in various  analogous types of cubature rules for the efficient numerical integration of functions in several variables
\cite{coo:constructing,coo-mys-sch:cubature,dun-xu:orthogonal,sob:cubature,sob-vas:theory,str:approximate}. In this note we are mainly
 concerned with a particular class of cubature rules that arises from the theory of \emph{symmetric} orthogonal polynomials
\cite{ber-sch-xu:multivariate,bru-hri-mot:discrete,hri-mot:discrete,hri-mot-pat:cubature,li-xu:discrete,moo-mot-pat:gaussian,moo-pat:cubature}. 
 These rules are designed to integrate symmetric functions over a hypercube.
The  cubatures in question provide a multivariate analog of  well-known Chebyshev type quadratures for the integration rational functions with prescribed poles outside the interval of integration \cite{bul-cru-dec-gon:rational,dar-gon-jim:quadrature,dec-van-bul:rational,die-ems:quadrature,van-bul-gon:computing}. 

Specifically, we will derive a cubature rule for integrals of rational symmetric functions in  $\cos (\xi_1),\ldots ,\cos (\xi_n)$ that admit poles
on prescribed coordinate hyperplanes outside the integration domain. 
The integration measures stem from the Haar measures of the orthogonal and the compact symplectic Lie groups (cf. e.g.
 \cite[Chapter 11.10]{pro:lie} and  \cite[Chapter IX.9]{sim:representations}).  They are given explicitly by the
unnormalized densities of  unitary Jacobi ensembles  (cf. e.g.  \cite[Sections 2.6, 3.7]{for:log-gases} and  \cite[Chapter 19]{meh:random}):
  \begin{align}\label{rho-jacobi}
\rho_\epsilon  (\boldsymbol{\xi}):=\prod_{1\leq j\leq n} 2^{\epsilon_++\epsilon_-} & \bigl(1+\epsilon_+\cos(\xi_j)\bigr)   \bigl(1-\epsilon_-\cos(\xi_j)\bigr) \\
&\times \prod_{1\leq j<k\leq n}  \bigl(\cos (\xi_j)-\cos(\xi_k)\bigr)^2 ,\nonumber
\end{align}
where $\boldsymbol{\xi}=(\xi_1,\ldots ,\xi_n)\in [ 0,\pi ]^n$ and $\epsilon_+ ,\epsilon_-\in \{ 0,1\}$. Here the different choices of  $\epsilon_+ ,\epsilon_-\in \{ 0,1\}$ correspond to
selecting the Haar measure on the orthogonal group in odd dimensions
$O(2n+1;\mathbb{R})$
(type $B_n$: $\epsilon_+ \neq \epsilon_-$), on the compact symplectic group $Sp (n;\mathbb{H})$ (type $C_n$: $\epsilon_\pm =1$), or on the orthogonal group in even dimensions $O(2n;\mathbb{R})$ (type $D_n$: $\epsilon_\pm =0$), respectively. The measures in question should be viewed as the natural analogs of the celebrated Chebyshev orthogonality measures in the context of the integration of symmetric functions.

To achieve the above-stated goal we present  in Section \ref{sec2} a simple formalism  that permits to convert quadrature rules into cubature formulas for symmetric functions.  It is based on the Cauchy-Binet formula in combination with its integral variant due to Andr\'eief \cite{and:note,for:meet}. For completeness, both formulas  are recalled in Appendix \ref{app} at the end of this note.  In the case of (quasi-)Gaussian
quadratures, the formalism of Section \ref{sec2} recovers a fundamental cubature rule originally found by Berens, Schmid and Xu \cite{ber-sch-xu:multivariate}. A key feature of the present approach based on the Cauchy-Binet-Andr\'eief formulas is that it is very straightforward to process input from a wide class of starting quadratures, including e.g.  those of Radau--  and Lobatto type.
Next we focus, in Section \ref{sec3}, on  quadratures stemming from the Bernstein-Szeg\"o polynomials  \cite{die-ems:quadrature}. These quadrature rules are designed to integrate rational functions with prescribed poles against the Chebyshev  weight functions. When lifting these quadratures following the procedure of Section \ref{sec2}, the desired cubature rules are found for the integration of
symmetric rational functions with prescribed poles on coordinate hyperplanes against the densities of the unitary Jacobi ensembles.
In the special case of  the Gauss quadrature associated with the  Bernstein-Szeg\"o polynomials \cite{bul-cru-dec-gon:rational,dar-gon-jim:quadrature,dec-van-bul:rational,van-bul-gon:computing},  we recover a corresponding specialization of the
Gaussian cubature from \cite{ber-sch-xu:multivariate} exhibited in \cite{die-ems:exact}.
Finally, when poles are absent our cubature formulas reduce  to more elementary cubature rules  for the integration of symmetric polynomials. The latter rules arose
previously in \cite{bru-hri-mot:discrete} and  \cite{hri-mot:discrete} via their relations with the discrete sine and cosine transforms.

\section{Cubature of symmetric polynomials by lifting  quadrature rules}\label{sec2}
Let $\text{w}(x)$, $a<x<b$ be a  weight function with finite moments $\int_a^b x^k \text{w}(x) \text{d} x$ ($k=0,1,2,\ldots$).
For our purposes it is enough to assume that $a,b\in\mathbb{R}$, but Proposition \ref{L-quadrature:prp} and its proof below actually remain
valid when considering $a=-\infty$ and/or $b=+\infty$, provided all the moment integrals converge in absolute value. For $m$ nonnegative integral, we denote by
\begin{subequations}
\begin{equation}\label{quadrature}
\int_a^b  f(x) \text{w}(x) \text{d} x  =   \sum_{0\leq l\leq m}    f( x^{(m+1)}_l)  \text{w}^{(m+1)}_{l}
\end{equation}
a quadrature rule with Christoffel weights $\text{w}^{(m+1)}_0,\ldots , \text{w}^{(m+1)}_m$ supported on $m+1$ nodes that  lie on  the real axis  within the domain of integration:
\begin{equation}\label{nodes}
a\leq x_0^{(m+1)}< x_1^{(m+1)}<\cdots <x_m^{(m+1)}\leq b.
\end{equation}
\end{subequations}
Throughout it will be assumed that the quadrature rule under consideration is exact for all polynomials $f(x)$ in $x$ of degree at most
\begin{equation*}
\texttt{D}=2m+1-\delta 
\end{equation*}
for some fixed ($m$-independent, integer-valued) constant $\delta\geq 0$. The constant in question measures the distance of the degree of exactness $\texttt{D}$ from
the optimal Gaussian value $2m+1$.  In other words, our rule takes effect as soon as
the number of nodes is sufficiently large so as to guarantee that
 $2m+1\geq \delta$ (so $\texttt{D}\ge 0$).

Let $\mathbb{P}^{(\texttt{D},n)}$ denote the $\binom{\texttt{D}+n}{n}$-dimensional space of symmetric polynomials in $\mathbf{x}:=(x_1,\ldots ,x_n)$ that are of degree at most $\texttt{D}$ in each of the variables $x_j$ ($j=1,\ldots ,n$). The following proposition lifts the quadrature rule in Eqs. \eqref{quadrature}, \eqref{nodes} to an exact cubature rule in  $\mathbb{P}^{(\texttt{D},n)}$.

\begin{proposition}[Exact Cubature Rule in $\mathbb{P}^{(\texttt{D},n)}$]\label{L-quadrature:prp}
For $f(\mathbf{x})\in\mathbb{P}^{(\emph{\texttt{D}},n)}$, one has that
\begin{subequations}
\begin{equation}\label{L-quadrature:a}
\frac{1}{n!} \int_a^{b}\cdots\int_a^b 
f (\mathbf{x})  \emph{W}^{(n)}(\mathbf{x}) \emph{d} x_1\cdots\emph{d} x_n 
 =\sum_{{\lambda}\in \Lambda^{(m,n)}}
f \bigl(\mathbf{x}^{(m,n)}_{{\lambda}}\bigr)   \emph{W}^{(m,n)}_{{\lambda}} ,
\end{equation}
where
\begin{equation}\label{L-quadrature:b}
 \emph{W}^{(n)}(\mathbf{x}):= \prod_{1\leq j<k\leq n} (x_j-x_k)^2      \prod_{1\leq j\leq n} \emph{w}(x_j) ,
\end{equation}
\begin{equation}\label{L-quadrature:c}
\mathbf{x}^{(m,n)}_{{\lambda}}:= \left( x^{(m+n)}_{{\lambda}_1+n-1},x^{(m+n)}_{{\lambda}_2+n-2},\ldots ,x^{(m+n)}_{{\lambda}_{n-1}+1},x^{(m+n)}_{{\lambda}_n}\right) ,
\end{equation}
\begin{equation} \label{L-quadrature:d}
 \emph{W}^{(m,n)}_{{\lambda}} := \prod_{1\leq j<k\leq n} \Bigl(x^{(m+n)}_{\lambda_j+n-j}-x^{(m+n)}_{\lambda_k+n-k} \Bigr)^2  \prod_{1\leq j\leq n}  \emph{w}^{(m+n)}_{{\lambda}_j+n-j} ,
\end{equation}
and the summation is over   $\binom{m+n}{n}$ nodes $\mathbf{x}^{(m,n)}_{{\lambda}}$  labeled by
\begin{equation}\label{L-quadrature:e}
\Lambda^{(m,n)}:=\{ \lambda =(\lambda_1,\ldots,\lambda_n) \in\mathbb{Z}^n \mid m \geq \lambda_1\geq \lambda_2\geq\cdots \geq \lambda_n\geq 0 \} .
\end{equation}
\end{subequations}
\end{proposition}
\begin{proof}
By linearity, it is sufficient to verify the cubature rule on the Schur basis $s_\mu (\mathbf{x})$, $\mu\in \Lambda^{(\texttt{D},n)}$ of $\mathbb{P}^{(\texttt{D},n)}$, where (cf. e.g. \cite{mac:symmetric})
\begin{equation*}
s_\mu (\mathbf{x}):=  \frac{ \det \left[ x_k^{\mu_j+n-j} \right]_{1\leq j,k\leq n} }{ \det \left[ x_k^{n-j} \right]_{1\leq j,k\leq n} } .
\end{equation*}
Elementary manipulations readily confirm  the validity of our cubature rule in this situation:
\begin{align*}
&\frac{1}{n!} \int_a^{b}\cdots\int_a^b 
s_\mu (\mathbf{x})  \text{W}^{(n)}(\mathbf{x}) \text{d} x_1\cdots\text{d} x_n \\
&\stackrel{(i)}{=} \frac{1}{n!} \int_a^{b}\cdots\int_a^b 
 \det \left[ x_k^{\mu_j+n-j} \right]_{1\leq j,k\leq n}  \det \left[ x_k^{n-j} \right]_{1\leq j,k\leq n}
\prod_{1\leq j\leq n} \text{w}(x_j)\,
 \text{d} x_1\cdots\text{d} x_n  \\
 &\stackrel{(ii)}{=} \det  \left[    \int_a^b   x^{\mu_j+2n-j-k}  \text{w}(x) \text{d}x  \right]_{1\leq j,k\leq n} \\
 &\stackrel{(iii)}{=}  \det  \left[    \sum_{0\leq l< m+n}     \bigl(x^{(m+n)}_l\bigr)^{\mu_j+2n-j-k}   \text{w}^{(m+n)}_{l} \right]_{1\leq j,k\leq n}  \\
 &\stackrel{(iv)}{=}\sum_{m+n>l_1>l_2>\cdots >l_n\geq 0}  \Biggl( 
 \det  \left[ \bigl(x^{(m+n)}_{l_k}   \bigr)^{\mu_j+n-j} \right]_{1\leq j,k\leq n} \\  
& \hspace{12em} \times \det \left[ \bigl(x^{(m+n)}_{l_k}  \bigr)^{n-j} \right]_{1\leq j,k\leq n}
\prod_{1\leq j\leq n} \text{w}^{(m+n)}_{l_j} \Biggr)  \\
&\stackrel{(v)}{=}\sum_{{\lambda}\in \Lambda^{(m,n)}}
s_\lambda \bigl(\mathbf{x}^{(m,n)}_{{\lambda}}\bigr)   \text{W}^{(m,n)}_{{\lambda}} .
\end{align*}
In the successive steps above we used: $(i)$ the Vandermonde determinant 
\begin{equation*}  \det \left[ x_k^{n-j} \right]_{1\leq j,k\leq n}= \prod_{1\leq j<k\leq n} (x_j-x_k) ,\end{equation*}
$(ii)$ Andr\'eief's integral counterpart of the Cauchy-Binet formula (cf. Eq. \eqref{andreief} in Appendix \ref{app} with $f_j(x)=x^{\mu_j+n-j}$ and $g_j(x)=x^{n-j}$), $(iii)$ the quadrature rule from Eqs. \eqref{quadrature}, \eqref{nodes}
on $m+n$ nodes, $(iv)$ the Cauchy-Binet formula (cf. Eq. \eqref{cauchy-binet} in Appendix \ref{app}), and $(v)$ the Vandermonde determinant.
\end{proof}

The quadrature rule in Eqs. \eqref{quadrature}, \eqref{nodes}  is  {\em interpolatory}  if it is exact for (Lagrange) interpolation polynomials on the nodes. This is the case when  $m+1\geq \delta$ (so $\texttt{D}\geq m$). Moreover, the rule
 is {\em positive} 
if both $\text{w}(x)>0$ for $a<x<b$ and $\text{w}^{(m+1)}_l>0$ for $0\leq l\leq m$. For positive interpolatory quadratures of degree $ \texttt{D}\geq 2m$ (so $\delta=0$ or  $\delta= 1$) the formula in Proposition \ref{L-quadrature:prp} goes back to \cite{ber-sch-xu:multivariate}, where it was deduced with the aid of multivariate orthogonal polynomials associated with
the weight function $W^{(n)}(\mathbf{x})$  \eqref{L-quadrature:b}.  In this situation the underlying
quadrature is Gaussian if $\texttt{D}=2m+1$ (i.e. $\delta=0$), while for $\texttt{D}=2m$ (i.e. $\delta=1$)
one is dealing with a quasi-Gaussian quadrature if all nodes belong to the open interval $]a,b[$ (cf. e.g. \cite{mic-riv:numerical,xu:characterization}) and with a Gauss-Radau quadrature if {\em one} of the nodes 
attains the boundary of this interval (cf. e.g. \cite{gau:survey}). Let us recall in this connection  that when  both boundary points are attained by the nodes, then the optimal degree of exactness is achieved by the corresponding Gauss-Lobatto quadrature at $\texttt{D}=2m-1$ (i.e. $\delta=2$), cf. e.g. \cite{gau:survey}.

\section{Cubature rule for unitary Jacobi ensembles}\label{sec3}
In \cite[Theorem 5]{die-ems:quadrature} we presented a positive quadrature rule of the form
\begin{subequations}
\begin{equation}\label{q-rule:a}
\frac{1}{2\pi} \int_0^\pi  R (\xi)  \rho_\epsilon  (\xi) \text{d}\xi =
 \sum_{0\leq {l}\leq m}   R\bigl( \xi _{{l}}^{(m+1)}\bigr)  \rho_\epsilon \bigl(\xi _{l}^{(m+1)}\bigr) {\Delta}^{(m+1)}_{{l}} ,
\end{equation}
for the integration of rational functions $R(\cdot )$ in $\cos (\xi)$ against the Chebyshev weight functions
\begin{equation}\label{q-rule:b}
\rho_\epsilon (\xi):=2^{\epsilon_++\epsilon_-}(1+\epsilon_+\cos(\xi))   (1-\epsilon_-\cos(\xi)) \qquad  (\epsilon_\pm \in \{ 0,1\} ).
\end{equation}
In this formula the positions of the nodes
\begin{equation}\label{q-rule:c}
0\leq \xi_0^{(m+1)}< \xi_1^{(m+1)}<\cdots <\xi_m^{(m+1)}\leq \pi
\end{equation}
 are controlled by two families of parameters $a_1,\ldots ,a_d$ and $\tilde{a}_1,\ldots,\tilde{a}_{\tilde{d}}$ with $|a_r|<1$ and $|\tilde{a}_r|<1$. 
 Here and below these parameters are allowed to be complex, with the assumption that both the sets
 $\{ a_1,\ldots ,a_d\}$ and $\{\tilde{a}_1,\ldots,\tilde{a}_{\tilde{d}}\}$ are invariant under complex conjugation. Specifically, the node $\xi_{l}^{(m+1)}$  ($l\in \{ 0,\ldots ,m\}$) is retrieved as the unique real solution of the transcendental equation
\begin{equation}\label{bethe:eq}
2\bigl(m-d_\epsilon-\tilde{d}_{\tilde{\epsilon}}\bigr)\xi + \sum_{1\leq r\leq d}  \int_0^\xi u_{a_r}(\theta)\text{d}\theta +    
 \sum_{1\leq r\leq \tilde{d}}  \int_0^\xi u_{\tilde{a}_r}(\theta)\text{d}\theta  =\pi (2l+\epsilon_- +\tilde{\epsilon}_-) ,
\end{equation}
where
\begin{align}\label{ua}
u_a(\theta) :=&   \frac{1-a^2}{1-2a\cos (\theta)+a^2}  \qquad   (  |a |<1,\, \theta\in\mathbb{R})  , 
\end{align}
and the corresponding Christoffel weight is given by
\begin{align}\label{cf-weights}
 {\Delta}^{(m+1)}_{{l}}& :={\textstyle \left( \frac{1}{2}\right) ^{(1-\epsilon_-)(1-\tilde{\epsilon}_-)\delta_l+(1-\epsilon_+)(1-\tilde{\epsilon}_+)\delta_{m-l}}}  \times\\
&  \Biggl(2\bigl(m-d_\epsilon-\tilde{d}_{\tilde{\epsilon}}\bigr) + 
\sum_{1\leq r\leq d} u_{a_r}\bigl(\xi^{(m+1)}_{{l}}\bigr) +    
 \sum_{1\leq r\leq \tilde{d}}  u_{\tilde{a}_r}\bigl(\xi^{(m+1)}_{{l}} \bigr) \Biggr)^{-1} . \nonumber
\end{align}
Here $d_\epsilon:= \frac{1}{2} (d-\epsilon_+-\epsilon_-)$, $\tilde{d}_{\tilde{\epsilon}}:= \frac{1}{2} (\tilde{d}-\tilde{\epsilon}_+-\tilde{\epsilon}_-)$,  and
\begin{equation*}
\delta_l:= \begin{cases}
1 &\text{if}\ l=0,\\
0&\text{otherwise}.
\end{cases} 
\end{equation*}
It is furthermore assumed that
$m$ is positive such that
\begin{equation}\label{m-condition}
{m> \lceil d_\epsilon \rceil + \lceil \tilde{d}_{\tilde{\epsilon}}\rceil } .
\end{equation}
\end{subequations}
The quadrature in Eqs. \eqref{q-rule:a}--\eqref{m-condition} is exact for rational functions $R(\cdot)$  with prescribed poles of the form
\begin{subequations}
\begin{equation}\label{q-rule:r}
R(\xi )  =   \frac{f\bigl(\cos (\xi)\bigr)}{\prod_{1\leq r\leq d} \bigl(1-2 a_r\cos (\xi )+ a_r^2\bigr)}   ,
\end{equation}
where $f\bigl(\cos(\xi)\bigr)$ stands for an arbitrary polynomial   in $\cos (\xi)$ of degree at most
\begin{equation}\label{q-rule:d}
\texttt{D}= 2m+\tilde{\epsilon}_++\tilde{\epsilon}_- -\tilde{d}-1.
\end{equation}
\end{subequations}

With the aid of Proposition \ref{L-quadrature:prp}, we will now lift the quadrature in question to a corresponding cubature rule for the integration of
symmetric functions---with prescribed poles at coordinate hyperplanes---against the densities of the unitary Jacobi ensembles.

\begin{theorem}[Cubature Rule for Unitary Jacobi Ensembles]\label{jacobi-cubature:thm}
Let $\epsilon_\pm,\tilde{\epsilon}_\pm \in \{ 0, 1\}$,  $|a_r|<1$ ($r=1,\ldots ,d$),  $|\tilde{a}_r|<1$ ($r=1,\ldots ,\tilde{d}$) with in each case (possible) complex parameters arising in complex conjugate pairs. Then assuming
\begin{subequations}
\begin{equation}\label{jacobi-cubature:a}
m+n -1>  \lceil d_\epsilon \rceil + \lceil \tilde{d}_{\tilde{\epsilon}}\rceil ,
\end{equation}
one has that 
\begin{align}\label{jacobi-cubature:b}
\frac{1}{(2\pi )^n\, n!}&
\int_0^\pi\cdots \int_0^\pi  R ( \boldsymbol{\xi})  \rho_\epsilon  (\boldsymbol{\xi}) \text{d}\xi_1\cdots \text{d}\xi_n =\\
& \sum_{{\lambda}\in\Lambda^{(m,n)}}   R\bigl( \boldsymbol{ \xi} _{{\lambda}}^{(m,n)} \bigr)  \rho_\epsilon \bigl(\boldsymbol{ \xi} _{{\lambda}}^{(m,n)} \bigr)  \Delta_\lambda^{(m,n)} 
.\nonumber
\end{align}
Here $\rho_\epsilon  (\boldsymbol{\xi})$ is given by the unitary Jacobi distribution in  Eq. \eqref{rho-jacobi}, the nodes $\boldsymbol{\xi}^{(m,n)}_{{\lambda}}$ are of the form in Eq. \eqref{L-quadrature:c} (with $x$ replaced by $\xi$), the Christoffel weights read
\begin{equation}
\Delta_\lambda^{(m,n)} := \prod_{1\leq j\leq n}   \Delta_{\lambda_j+n-j}^{(m+n)} ,
\end{equation}
and $R(\cdot)$ is of the form
\begin{equation}
R(\boldsymbol{\xi})  =   \frac{f\bigl(\cos (\xi_1),\ldots ,\cos(\xi_n)\bigr)}{\prod_{\substack{1\leq r\leq d\\ 1\leq j\leq n}} \bigl(1-2a_r\cos (\xi_j )+a_r^2\bigr)}    ,
\end{equation}
 where $f(x_1,\ldots,x_n)=f(\mathbf{x})$ denotes an arbitrary symmetric polynomial in $\mathbb{P}^{(\emph{\texttt{D}},n)}$ with $\emph{\texttt{D}}$ taken from Eq. \eqref{q-rule:d}.
 \end{subequations}
\end{theorem}

\begin{proof}
After a change of variable of the form $x=-\cos(\xi)$ the quadrature rule in Eqs. \eqref{q-rule:a}--\eqref{m-condition}---for $R(\cdot)$ from Eq. \eqref{q-rule:r} (subject to the constraint
\eqref{q-rule:d} on the maximal degree $\texttt{D}$ of the polynomial in the numerator)---becomes of the standard form in Eqs. \eqref{quadrature}, \eqref{nodes} with
$(a,b)=(-1,1)$, $x_l^{(m+1)}=-\cos \bigl( \xi^{(m+1)}_l\bigr)$, and
\begin{align*}
\text{w}(x) &=  \frac{(1-\epsilon_+ x)(1+\epsilon_- x)}{2\pi \sqrt{1-x^2}  \prod_{1\leq r\leq d} (1+2a_r x+a_r^2)} , \\
\text{w}^{(m+1)}_l& =  \Delta_l^{(m+1)} 
 \frac{ (1-\epsilon_+ x^{(m+1)}_l)(1+\epsilon_- x^{(m+1)}_l) }{ \prod_{1\leq r\leq d} (1+2a_r x^{(m+1)}_l+a_r^2)}.
\end{align*}
Upon applying Proposition \ref{L-quadrature:prp} and transforming back to trigonometric variables $x_j=-\cos (\xi_j)$ ($j=1,\ldots ,n$), the asserted cubature rule follows.
\end{proof}

\begin{remark}\label{estimates:rem}
It is immediate  from the transcendental equation in Eqs. \eqref{bethe:eq}, \eqref{ua} (via the mean value theorem) that the solutions $\xi^{(m+n)}_0,\ldots , \xi^{(m+n)}_{m+n-1}$ building the cubature nodes $\boldsymbol{\xi}^{(m,n)}_\lambda$, $\lambda\in\Lambda^{(m,n)}$ satisfy the following inequalities
(cf. \cite[Section 4.2]{die-ems:quadrature}):
\begin{equation*}\label{bound:a}
\frac{\pi \bigl({l}+\frac{1}{2}(\epsilon_- +\tilde{\epsilon}_-) \bigr)}{m +n-1-{d_\epsilon}- \tilde{d}_{\tilde{\epsilon}}+\kappa_-}   \leq \xi_{{l}}^{(m+n)}  \leq 
\frac{\pi \bigl({l}+\frac{1}{2}(\epsilon_- +\tilde{\epsilon}_-) \bigr)}{m +n-1 -{d_\epsilon}-\tilde{d}_{\tilde{\epsilon}}+\kappa_+} 
\end{equation*}
(for $0\leq {l}< m+n$), and
\begin{equation*}\label{bound:b}
\frac{\pi ({k}-{l})}{m +n-1-{d_\epsilon}- \tilde{d}_{\tilde{\epsilon}}+\kappa_-}   \leq \xi_{{k}}^{(m+n)} - \xi_{{l}}^{(m+n)} \leq 
\frac{\pi ({k}-{l})}{m+n-1 -{d_\epsilon}- \tilde{d}_{\tilde{\epsilon}}+\kappa_+} 
\end{equation*}
(for $0\leq {l}<{k}< m+n$), where
\begin{equation*}\label{kappa}
\kappa_\pm := \frac{1}{2}\sum_{1\leq r\leq d}   \left(\frac{1- | a_r|}{1 + | a_r|}\right)^{\pm 1}   + 
  \frac{1}{2}\sum_{1\leq r\leq \tilde{d}}   \left(\frac{1- |\tilde{a}_{r|}}{1 + |\tilde{a}_{r|}}\right)^{\pm 1}  .
\end{equation*}
Moreover,  it is clear from  Eqs. \eqref{bethe:eq}, \eqref{ua} that the boundary value $\xi^{(m+n)}_0=0$
is attained iff $\epsilon_-=\tilde{\epsilon}_-=0$ and the boundary value
$ \xi^{(m+n)}_{m+n-1}=\pi$
is attained  iff $\epsilon_+=\tilde{\epsilon}_+=0$ (since $\int_0^\pi u_a (\theta) \text{d}\theta =\pi$ for $|a|<1$).
\end{remark}

\begin{remark}\label{gaussian:rem}
In the above integration formulas the degree of exactness  is optimal if $\texttt{D}$ \eqref{q-rule:d} reaches the Gaussian value $2m+1$, which is achieved when $\tilde{d}=0$ and $\tilde{\epsilon}_\pm =1$.  This
special case of the quadrature rule in Eqs. \eqref{q-rule:a}--\eqref{m-condition}  can be inferred from
\cite{dar-gon-jim:quadrature} for $\epsilon_\pm =0$, and  from
 \cite{bul-cru-dec-gon:rational,dec-van-bul:rational,van-bul-gon:computing} for general $\epsilon_\pm \in \{ 0,1\}$ (cf. also
 \cite[Section 8]{die-ems:exact}). The corresponding
specialization of the cubature rule in Theorem \ref{jacobi-cubature:thm} was presented in turn in \cite[Section 9]{die-ems:exact}.
\end{remark}

\begin{remark}
When $d=\tilde{d}=0$, Theorem \ref{jacobi-cubature:thm} reduces to an elementary cubature rule of the form
\begin{subequations}
\begin{align}\label{GCc:a}
&\frac{1}{(2\pi )^n\, n!}\int_0^\pi\cdots \int_0^\pi  f \bigl( \cos(\boldsymbol{\xi})\bigr)  \rho_\epsilon  (\boldsymbol{\xi}) \text{d}\xi_1\cdots \text{d}\xi_n =
 \\
& \frac{1}{N_\epsilon^{(m,n)}}
\sum_{{\lambda}\in\Lambda^{(m,n)}}  
{\textstyle \left( \frac{1}{2}\right) ^{(1-\epsilon_+)(1-\tilde{\epsilon}_+)\delta_{m-\lambda_1}+(1-\epsilon_-)(1-\tilde{\epsilon}_-)\delta_{\lambda_n}} }
 f \left( \cos \bigl( \boldsymbol{ \xi} _{{\lambda}}^{(m,n)} \bigr) \right) \rho_\epsilon \bigl(\boldsymbol{ \xi} _{{\lambda}}^{(m,n)} \bigr) \nonumber
\end{align}
for $ f \bigl( \cos(\boldsymbol{\xi})\bigr) := f\bigl(\cos(\xi_1),\ldots,\cos(\xi_n)\bigr)$, where $f(x_1,\ldots,x_n)=f(\mathbf{x})\in \mathbb{P}^{(\texttt{D},n)}$ with
$\texttt{D}=2m+\tilde{\epsilon}_+ +\tilde{\epsilon}_- -1$ and
\begin{equation}
N^{(m,n)}_\epsilon:= \bigl( 2(m+n-1)+\epsilon_+ +\epsilon_- + \tilde{\epsilon}_+ +\tilde{\epsilon}_-\bigr)^n .
\end{equation}
The corresponding cubature nodes
$ \boldsymbol{ \xi} _{{\lambda}}^{(m,n)}$, $\lambda\in\Lambda^{(m,n)}$  become in this situation explicit in closed form:
\begin{equation}\label{GCc:b}
\xi^{(m+n)}_{{l}}=
\frac
{\pi \bigl({l}+\frac{1}{2}(\epsilon_- +\tilde{\epsilon}_-) \bigr)}
{m+n-1+ \frac{1}{2}(\epsilon_+ +\epsilon_- + \tilde{\epsilon}_+ +\tilde{\epsilon}_-)} \qquad (0\le {l}<m+n).
\end{equation}
\end{subequations}
The case $\epsilon_-=0$ recovers (up to a change of variables) eight cubature formulas stemming from the discrete cosine transforms DCT-1,$\ldots$,DCT-8 detailed in
\cite[Section 5]{hri-mot:discrete} (cf. loc. cit. Eqs. (5.9), (5.10) and Subsections 5.3.2, 5.3.4). The case $\epsilon_-=1$ recovers in turn analogous cubature formulas stemming from the discrete sine transforms DST-1,$\ldots$,DST-8 \cite{bru-hri-mot:discrete} (cf. also Remark \ref{dxt:rem} below). When $\tilde{\epsilon}_\pm=1$ the cubature rule \eqref{GCc:a}--\eqref{GCc:b} is exact of optimal degree $\texttt{D}=2m+1$; this case
was highlighted in  \cite[Eqs. (9.2a), (9.2b)]{die-ems:exact} as a special elementary instance of the Gaussian cubature rule alluded to in Remark \ref{gaussian:rem} above.
\end{remark}

\begin{remark}\label{dxt:rem}
The quadrature in Eqs. \eqref{q-rule:a}--\eqref{m-condition} originates from a finite system of discrete orthogonality relations for the Bernstein-Szeg\"o polynomials  \cite[Section 2.3]{die-ems:quadrature}.
When all parameters
$a_1,\ldots,a_d$ and $\tilde{a}_1,\ldots ,\tilde{a}_{\tilde{d}}$ vanish, the orthogonality relations in question
simplify and express the orthogonality of a matrix representing the kernel of a discrete trigonometric transform:
\begin{equation*}\label{DXT}
\Psi^{(m+1)}:=\left[\psi^{(m+1)}_{l,k} \right]_{0\leq l,k\leq m} ,
\end{equation*}
where
\begin{align*}
 \psi^{(m+1)}_{l,k}&:=\sqrt{\frac{2}{m+ \frac{1}{2}(\epsilon_+ +\epsilon_- + \tilde{\epsilon}_+ +\tilde{\epsilon}_-)} }
\, \text{cs}_\epsilon \left(   \frac{\pi   \bigl(l+ \frac{1}{2}(\epsilon_-  +\tilde{\epsilon}_-)\bigr)  \bigl( k +  \frac{1}{2}(\epsilon_- +\epsilon_+ )\bigr) }{m+ \frac{1}{2}(\epsilon_+ +\epsilon_- + \tilde{\epsilon}_+ +\tilde{\epsilon}_-)}     \right)
\\
&\times \left( \frac{1}{\sqrt{ 2}}\right)^{   (1-\epsilon_-)(1-\tilde{\epsilon}_-)\delta_l +(1-\epsilon_+)(1-\tilde{\epsilon}_+) \delta_{m-l}  +  (1-\epsilon_-)(1-\epsilon_+) \delta_k+ (1-\tilde{\epsilon}_-)(1-\tilde{\epsilon}_+)\delta_{m-k}         },
\end{align*}
with 
\begin{equation*}
\text{cs}_\epsilon (\xi):= \begin{cases}
\cos(\xi) &\text{if}\ \epsilon_-=0 ,\\
\sin(\xi) &\text{if}\ \epsilon_-=1.
\end{cases} 
\end{equation*}
It is well-known that the kernels of such discrete trigonometric transforms arise from diagonalizing  Jacobi matrices of the form
\begin{equation*}\label{J}
\text{J}^{(m+1)}=
\begin{bmatrix}
b  & \sqrt{a} & & & &\\
\sqrt{a} &  0 & 1& & & \\
 & 1 &\ddots &\ddots & &\\
 &  &\ddots &\ddots &1 &\\
& &   &  1       &0&\sqrt{\tilde{a}} \\
 & & & &  \sqrt{\tilde{a}}  &\tilde{b} \end{bmatrix} ,
\end{equation*}
with $a:=2^{(1-\epsilon_-)(1-\epsilon_+)} $, $b:=\epsilon_+- \epsilon_- $ and $\tilde{a}:=2^{(1-\tilde \epsilon_-)(1-\tilde\epsilon_+)}$, $\tilde{b}:= \tilde{\epsilon}_+ -  \tilde{\epsilon}_-$. Indeed, one has that
(cf. e.g. \cite{bri-yip-rao:discrete,die-ems:discrete,die-ems:quadrature,str:discrete})
\begin{equation*}
\Psi^{(m+1)} \text{J}^{(m+1)} =\text{E}^{(m+1)}  \Psi^{(m+1)}  ,
\end{equation*}
where 
\begin{equation*}
 \text{E}^{(m+1)} := \text{diag} \left[  2\cos\bigl( \xi^{(m+1)}_0\bigr),\ldots ,2\cos\bigl(\xi^{(m+1)}_m\bigr) \right] ,
\end{equation*}
with
$\xi^{(m+1)}_0,\ldots , \xi^{(m+1)}_m$ taken from Eq. \eqref{GCc:b} (so $n=1$). The  following table encodes how to indentify in the above formulae the kernels (and the Jacobi matrices)
corresponding to the sixteen standard discrete trigonometric transforms; this yields
the discrete cosine transforms DCT-1,$\ldots$,DCT-8 and the discrete sine transforms DST-1,$\ldots$,DST-8, respectively, upon specializing  the boundary parameters $\epsilon_\pm,\tilde{\epsilon}_\pm \in \{ 0,1\}$. 
\begin{equation*} \label{DXT:table}
\begin{tabular}{@{}|c|cccc|c|@{}} 
\toprule
$(\tilde \epsilon_-,\tilde \epsilon_+) \backslash (\epsilon_-, \epsilon_+)$ & (0,0) & (0,1)& (1,0) & (1,1)\\
\midrule
(0,0) & DCT-1 & DCT-6 &  DST-8  &  DST-3 \\ 
(0,1) & DCT-5 & DCT-2 &  DST-4  & DST-7 \\ 
(1,0) & DCT-7 & DCT-4 &  DST-2  & DST-5 \\ 
(1,1) & DCT-3 & DCT-8 & DST-6   &  DST-1 \\ 
\bottomrule
\end{tabular} 
\end{equation*}
The underlying orthogonal families, which arise as parameter specializations of the Bernstein-Szeg\"o polynomials, are given by the Chebyshev polynomials of the first kind ($\epsilon_\pm =0$), of the second kind ($\epsilon_\pm =1$), of the third kind ($\epsilon_+=0$, $\epsilon_-=1$),
and of the fourth kind  ($\epsilon_+=1$, $\epsilon_-=0$), respectively (cf. \cite[Remark 6.1]{die-ems:exact}).
\end{remark}

\appendix

\section{The Cauchy-Binet-Andr\'eief formulas}\label{app}
Let $f_1,\ldots f_n$ and $g_1,\ldots ,g_n$ be functions in the Hilbert space $L^2\bigl( ]a,b[ , | \text{w}(x) | \text{d}x\bigr)$. Then  Andr\'eief's integration formula states
that  \cite{and:note}:
\begin{align}\label{andreief}
\frac{1}{n!} \int_a^{b}\cdots\int_a^b &
 \det \left[ f_j(x_k)  \right]_{1\leq j,k\leq n}  \det \left[ g_j(x_k)\right]_{1\leq j,k\leq n}
\prod_{1\leq j\leq n} \text{w}(x_j)\,
 \text{d} x_1\cdots\text{d} x_n  \nonumber \\
 &=\det  \left[    \int_a^b   f_j(x)g_k(x)  \text{w}(x) \text{d}x  \right]_{1\leq j,k\leq n} .
\end{align}
 A short and elementary verification of this identity is provided e.g. in \cite[Lemma 3.1]{bai-dei-stra:products}. A similar proof can be found in \cite{for:meet} together with a historical account of the formula. 
If we pass from functions on the interval to functions supported on the nodes $x_0^{(m+n)}<\cdots <x_{m+n-1}^{(m+n)} $, and replace the
integration measure $ \text{w}(x)\text{d}x$ by the discrete point measure on these nodes with weights  $\text{w}_0^{(m+n)},\ldots ,\text{w}_{m+n-1}^{(m+n)}$, then
 Andr\'eief's integration formula gives rise to the identity
 \begin{align*}
 \frac{1}{n!}\sum_{\substack{0\leq l_k <m+n \\ k=1,\ldots ,n}} &  \Biggl( 
 \det  \left[ f_j\bigl(x^{(m+n)}_{l_k}   \bigr)  \right]_{1\leq j,k\leq n}  \det \left[ g_j \bigl(x^{(m+n)}_{l_k}  \bigr) \right]_{1\leq j,k\leq n}
\prod_{1\leq j\leq n} \text{w}^{(m+n)}_{l_j} \Biggr)  \\
 &=\det  \left[    \sum_{0\leq l< m+n}    f_j \bigl(x^{(m+n)}_l\bigr)   g_k \bigl(x^{(m+n)}_l\bigr)   \text{w}^{(m+n)}_{l} \right]_{1\leq j,k\leq n}  .
\end{align*}
Upon exploiting the anti-symmetry of the determinants with respect to the ordering of  the indices $l_1,\ldots ,l_n$,  the LHS  can be rewritten as
\begin{equation*}
\sum_{m+n>l_1>l_2>\cdots >l_n\geq 0}  \Biggl( 
 \det  \left[ f_j\bigl(x^{(m+n)}_{l_k}   \bigr)  \right]_{1\leq j,k\leq n} 
 \det \left[ g_j \bigl(x^{(m+n)}_{l_k}  \bigr) \right]_{1\leq j,k\leq n}
\prod_{1\leq j\leq n} \text{w}^{(m+n)}_{l_j} \Biggr)  .
\end{equation*}
This shows that the  identity of interest amounts to the celebrated Cauchy-Binet formula
\begin{equation}\label{cauchy-binet}
\sum_{\substack{L\subset \{ 0,\ldots,m+n-1\} \\ |L| = n}} \det F_L \det G_L = \det \left( F G^T \right) ,
\end{equation}
with $F=F_{\{ 0,\ldots,m+n-1\}}$,  $G=G_{\{ 0,\ldots,m+n-1\}}$ and
\begin{equation*}
F_L=  \left[ f_j\bigl(x^{(m+n)}_{l}   \bigr)  \right]_{\substack{1\leq j \leq n\\ l\in L} }     ,  \qquad
G_L=  \left[ g_j\bigl(x^{(m+n)}_{l}   \bigr)   \text{w}^{(m+n)}_{l}  \right]_{\substack{1\leq j \leq n\\ l\in L} }  .
\end{equation*}
In Eq. \eqref{cauchy-binet} the sum is over all subsets  of indices $L$ of cardinality $|L|=n$, and 
 $G^T$ refers to the transposed matrix.

\bibliographystyle{amsplain}

\end{document}